\documentclass[reqno, a4paper, 10pt]{amsart}
\usepackage[utf8]{inputenc}
\usepackage{graphicx}
\usepackage{url}
\usepackage{mathrsfs}
\usepackage{color}
\usepackage[colorlinks,bookmarksopen,bookmarksnumbered,citecolor=red,urlcolor=red]{hyperref}

\newcommand{\tf}{t_\mathrm{f}}
\newcommand{\Ut}{U_{[t_0,\tf]}}
\newcommand{\Vt}{V_{[t_0,\tf]}}
\newcommand{\Wt}{W_{[t_0,\tf]}}

\newcommand{\Vcal}{\mathcal{V}}
\newcommand{\Sscr}{\mathscr{S}}

\renewcommand{\:}{\mathcal{\colon}}
\newcommand{\NN}{\mathbb{N}}

\newcommand{\RR}{\mathbb{R}}

\newcommand{\laplace}{\Delta}

\newcommand{\Lagr}{\mathcal{L}}

\DeclareMathOperator{\esssup}{ess\,sup}
\newcommand{\CQref}{\text{(CQ)}}

\begin{document}

\title[On the Value Function in Mixed-Integer Optimal Control]{Lipschitz Continuity of the Value Function in Mixed-Integer Optimal Control Problems$^*$}

\author{Martin Gugat$^\dag$,~Falk M. Hante$^\dag$}

\date{December 29, 2016}

\keywords{Parametric optimal control; parametric switching control; parametric optimization;
sensitivity; mixed-integer optimal control problems; optimal value function; Lipschitz
continuity}

\thanks{$^*$~The article is published in \emph{Math. Control Signals Syst.} (2017) 29:3.\\
\hspace*{\parindent}$^\dag$~Lehrstuhl für Angewandte Mathematik 2, Department Mathematik, Friedrich-Alexander Universität Erlangen-Nürnberg, \url{{falk.hante,martin.gugat}@fau.de}.}

%%%%%%%%%%%%%%%%%%%%%%%%%%%%%%%%%%%%%%%%%%%%%%%%%%%%%%%%%%%%%%%%%%%%%%%%%%%%%%%%
\begin{abstract}
We study the optimal value function for control problems on Banach spaces
that involve both continuous and discrete control decisions. 
For problems involving semilinear dynamics subject to mixed control 
inequality constraints, one can show that the optimal value depends 
locally Lip\-sch\-itz con\-ti\-nuous\-ly on perturbations of the 
initial data and the costs under rather natural assumptions. We prove 
a similar result for perturbations of the initial data, the constraints and the costs 
for problems involving linear dynamics, convex costs and convex constraints 
under a Slater-type constraint qualification. We show by an example that 
these results are in a sense sharp.
\end{abstract}
%%%%%%%%%%%%%%%%%%%%%%%%%%%%%%%%%%%%%%%%%%%%%%%%%%%%%%%%%%%%%%%%%%%%%%%%%%%%%%%%

\maketitle

%%%%%%%%%%%%%%%%%%%%%%%
\theoremstyle{plain}
\newtheorem{theorem}{Theorem}
\newtheorem{proposition}{Proposition}
\newtheorem{corollary}{Corollary}
\newtheorem{lemma}{Lemma}
\newtheorem{hypothesis}{Hypothesis}
\theoremstyle{remark}
\newtheorem{assumptions}{Assumption}
\newtheorem{definition}{Definition}
\newtheorem{remark}{Remark}
\newtheorem{example}{Example}
%%%%%%%%%%%%%%%%%%%%%%%

\section{Introduction}
In this paper we address the robustness of solutions to optimal control problems that
involve both continuous-valued and discrete-valued control decisions to steer
solutions of a differential equation such that an associated cost is minimized.
This problem class includes in particular optimal control of switched
systems \cite{Antsaklis2014,Zuazua2011}, but also optimization of systems with coordinated
activation of multiple actuators, for example, at different locations in space for certain
distributed parameter systems \cite{IftimeDemetriou2009,HanteSager2013}. In analogy to mixed-integer programming we call such
problems mixed-integer optimal control problems. Algorithms to compute solutions
to such problems are discussed in \cite{Gerdts2006,Sager2009,HanteSager2013,SastryEtAl2013a,SastryEtAl2013b,RuefflerHante2016}.
From a theoretical point of view, but also for a reliable application of such algorithms, the robustness
of the solution with respect to perturbation of data in the problem is essential, for instance,
in the case of uncertain initial data. We consider the robustness of the optimal value
because this is the criterion determining the control decision. Moreover, we understand robustness 
in the sense that we consider the regularity of the optimal value as a function of the problem parameters.

For continuous optimization problems many sensitivity results are available,
see \cite{bonnansshapiro,MR2421286}. In particular certain regularity assumptions
and constraint qualifications guarantee the continuity of the optimal value function,
see \cite{MR669727,gu:onesi}. In the context of mixed-integer programming, in general,
the main difficulty is that the admissible set consists of several connected components and
jumps in the optimal value as function of the problem parameters can occur if due to parameter
changes connected components of the feasible set vanish. In mixed-integer linear programming
with bounded feasible sets, the continuity of the value function is therefore equivalent
to existence of a Slater-point \cite{Williams1989}. For mixed-integer convex programs,
constraint qualifications are given in \cite{gu97} which yield the existence of one-sided
directional derivatives of the value function and hence its Lipschitz continuity.
For optimal control problems in general, it is well known that one cannot expect more
regularity of the optimal value function than Lipschitz continuity. The following example
is an adaption of a classical one saying that this is also true for integer, and hence
mixed-integer controlled systems.
\begin{example}\label{ex:nonsmooth}
For some $\tf>0$ and $\lambda \in \RR$, consider the problem
\begin{equation}
\left.\begin{array}{l}
 \text{minimize}~y(\tf)~\text{subject to}\\
 \quad \dot{y}(t)=v(t)\,y(t),~\text{for a.\,e.}~t \in (0,\tf),\quad y(0)=\lambda\\
 \quad y(t) \in \RR,~v(t) \in \{0,1\}~\text{for a.\,e.}~t \in (0,\tf).
\end{array}\right\}
\end{equation}
The optimal value function $\nu(\lambda)=\inf\{y(\tf;\lambda) : v \in L^\infty(0,\tf;\{0,1\})\}$ can easily be seen to be
\begin{equation*}
 \nu(\lambda)=\begin{cases}e^{\tf}\lambda, &~\lambda<0,\\ \lambda &~\lambda \geq 0, \end{cases}
\end{equation*}
which is Lipschitz continuous, but not differentiable in $\lambda=0$.
\end{example}

For semilinear mixed-integer optimal control problems, we show below that for parametric initial data
as in the example, local Lipschitz continuity of the optimal value function can indeed be guaranteed
for a rather general setting without imposing a Slater-type condition. Similar results are well-known in
the classical Banach or Hilbert space case without mixed control constraints
\cite{CannarsaFrankowska1992,BarbuDaPrato1983}.
Further, we analyze parametric control constraints and parametric cost functions for convex programs.
For this case, we formulate a Slater-type condition guaranteeing again the local Lipschitz continuity
of the optimal value function.
Finally, for convex programs, we can combine both results to obtain local Lipschitz continuity jointly
for parametric initial data, control constraints and cost functions.

\section{Setting and Preliminaries}
Let $Y$ be a Banach space, $U$ be a complete metric space,
$\Vcal$ be a finite set, and $f\: [t_0,\tf] \times Y \times U \times \Vcal \to Y$.
We consider the control system
\begin{equation}\label{eq:controlsys}
\dot{y}(t) = A y(t) + f(t,y(t),u(t),v(t)),~t \in (t_0,\tf)~\text{a.\,e.},
\end{equation}
where $[t_0,\tf]$ is a finite time horizon with $t_0<\tf$, $A\: D(A) \to Y$ is a generator of a
strongly continuous semigroup $\{T(t)\}_{t \geq 0}$ of bounded linear operators on $Y$, and
where $u\: [t_0,\tf] \to U$ and $v\: [t_0,\tf] \to \Vcal$ are two independent measurable control
functions. Throughout the paper we consider the Lebesgue-measure. Our main concern will be the
confinement that the control $v$ only takes values from a finite set. Without loss of generality,
we may identify $\Vcal$ with a set of integers $\{0,1,\ldots,N-1\}$ and, in analogy to mixed-integer
programming, we refer to \eqref{eq:controlsys} as a \emph{mixed-integer control system}, where $u$
represents ordinary controls and $v$ integer controls.  Let $\Ut$ be a Banach subspace of
measurable ordinary control functions $u\: [t_0,\tf] \to U$ and let $\Vt$ be the set of
measurable integer control functions $v\: [t_0,\tf] \to \Vcal$. By the assumed finiteness
of $\Vcal$ we actually have $\Vt=L^\infty(t_0,\tf;\Vcal)$.

Let $\Lambda$ be a Banach space and consider \eqref{eq:controlsys} subject to a parametric
initial condition
\begin{equation}\label{eq:initialcondition}
 y(t_0)=y_0(\lambda),
\end{equation}
where $y_0(\lambda)$ is an initial state in $Y$ parametrized by $\lambda \in \Lambda$.

The separation of the control in $u$ and $v$ and the inherent integer confinement of the latter
control lets us formulate parametric control constraints of the mixed form
\begin{equation}\label{eq:controlrestriction}
g_k^v(\lambda,u,t) \leq 0,~k=1,\ldots,M,~t \in [t_0,\tf]
\end{equation}
where $M\in\NN$ and, for every $v \in \Vt$, the functions $g_1^v,\ldots,g_M^v \: \Lambda \times \Ut \times [t_0,\tf] \to \RR$
are given. These constraints can for example model anticipating control restrictions, where a decision
represented by $v$ at an earlier time limits control decisions for $u$ at different times. We discuss an example
in Section~\ref{sec:example}. In cases without mixed control constraints, we set $M=0$.

\begin{definition}\label{def:solutionMICP} For fixed $\lambda \in \Lambda$,
let $\Wt(\lambda)$ denote the set of all admissible controls
\begin{equation}\label{defWT}
\begin{aligned}
 \Wt(\lambda):=\{&(u,v) \in \Ut \times \Vt : \\
 &g_k^v(\lambda,u,t) \leq 0,~k=1,\ldots,M,~t \in [t_0,\tf]\}. 
\end{aligned}
\end{equation}
Moreover, we say that $y\: [t_0,\tf] \to Y$
is a \emph{solution of the mixed-integer control system}
if there exists an admissible pair of controls $(u,v)\in\Wt(\lambda)$ such that
%ordinary control $u \in \Ut$ and an integer control $v \in \Vt$ such that
$y \in C([t_0,\tf];Y)$ satisfies the integral equation
\begin{equation}\label{eq:controlsysint}
y(t) = T(t-t_0)y(t_0) + \int_{t_0}^{t} T(t-s)f(s,y(s),u(s),v(s))\,ds,~t\in[t_0,\tf]
\end{equation}
and \eqref{eq:initialcondition} holds. Let $\Sscr_{[t_0,\tf]}(\lambda)$
denote the set of all such solutions $y$ defined on $[t_0,\tf]$. For any $y \in \Sscr_{[t_0,\tf]}(\lambda)$,
we denote by $y=y(\cdot;y_0(\lambda),u,v)$ the dependency of $y$ on $y_0(\lambda)$, $u$ and $v$ if needed.
\end{definition}
According to Definition~\ref{def:solutionMICP}, $\Sscr_{[t_0,\tf]}(\lambda)$ consists of the mild solutions
of equation~\eqref{eq:controlsys} and covers in an abstract sense many evolution problems involving linear partial
differential operators \cite{Pazy1983}. It particular, the mild solutions coincide with the usual concept of
weak solutions in case of linear parabolic partial differential equations on reflexive $Y$ with distributed control
where $A$ arises from a time-invariant variational problem \cite{BensoussanDaPratoDelfourMitter1992}. For an example,
see Section~\ref{sec:example}.

In conjunction with the mixed-integer control system we consider a cost function
$\varphi\: \Lambda \times C([t_0,\tf];Y) \times \Ut \times \Vt \to \RR \cup \{\infty\}$
and define the \emph{mixed-integer optimal control problem} with parameter $\lambda$ as
\begin{equation}\label{eq:miocp}
\left.\begin{array}{l}
 \text{minimize}~\varphi(\lambda,y,u,v) \; \text{subject to}\\
 \quad \dot{y}(t)  =   A \, y(t)  + f(t,\,y(t),\, u(t), \, v(t)),\; t \in (t_0,\, \tf) \;{\rm a.e.},\\
 %+ B u(t) + F(t,v(t)),
 \quad y(0)= y_0(\lambda),
 %~t \in (t_0,\,\tf),
 \\
 \quad g_k^v(\lambda,u,t)  \leq 0~\text{for all}~t \in [t_0,\tf],~k=1,\ldots,M,\\
 \quad y \in C([t_0,\tf];Y),~u \in \Ut,~v \in \Vt.
\end{array}
\right\}
\end{equation}
We will study the corresponding \emph{optimal value} $\nu(\lambda)\in \RR \cup \{\pm\infty\}$ given by
\begin{equation}\label{eq:defnu}
\begin{array}{l}
 \nu(\lambda)=\inf \bigl\{\varphi(\lambda,y,u,v) :\\
 \quad \dot{y}(t)  =   A \, y(t)  + f(t,\,y(t),\, u(t), \, v(t)),\; t \in (t_0,\, \tf) \;{\rm a.e.},\\
 %+ B u(t) + F(t,v(t)),
 \quad y(0)= y_0(\lambda),
 %~t \in (t_0,\,\tf),
 \\
 \quad g_k^v(\lambda,u,t)  \leq 0~\text{for all}~t \in [t_0,\tf],~k=1,\ldots,M,\\
 \quad y \in C([t_0,\tf];Y),~u \in \Ut,~v \in \Vt\bigr\}
\end{array}
\end{equation}
in its dependency on the parameter $\lambda$.

For the mixed-integer control system, we will impose the following assumptions.
\begin{assumptions}\label{ass:ControlSys} The map $f\: [t_0,\tf] \times Y \times U \times \{v\} \to Y$ is continuous for all $v \in \Vcal$.
Moreover, there exists a function $k \in L^1(t_0,\tf)$ such that for all $(u,v) \in \Wt$, $y_1,y_2 \in Y$ and for almost every $t \in (t_0,\tf)$
\begin{alignat*}{2}
 &\text{(i)}\qquad &&|f(t,y_1,u(t),v(t))-f(t,y_2,u(t),v(t))| \leq k(t)|y_1 - y_2|\\
 &\text{(ii)}\qquad &&|f(t,0,u(t),v(t))| \leq k(t).
\end{alignat*}
\end{assumptions}
In particular, under these assumptions, the integral in \eqref{eq:controlsysint} is well-defined in the
Lebesgue-Bochner sense and from the theory of abstract Cauchy problems \cite{Pazy1983} we obtain a solution
$y$ in $C([0,\tf];Y)$ for all $y_0 \in Y$, $u \in \Ut$ and $v \in \Vt$. Moreover, the strong continuity of
$T(\cdot)$ and the Gronwall inequality yield the following solution properties.

\begin{lemma}\label{lem:bounds} Under the Assumptions~\ref{ass:ControlSys}, there exist constants $\gamma \geq 0$ and $w_0 \geq 0$ such that
for all $\lambda_1,\lambda_2 \in \Lambda$, setting $y_i=y(\cdot;y_0(\lambda_i),u,v)\in\Sscr_{[t_0,\tf]}(\lambda_i)$ for $i \in \{1,2\}$, for all $t \in [t_0,\tf]$
it holds $\|T(t)\| \leq \gamma \exp(w_0(t-t_0 ))$,
\begin{equation}\label{eq:aprioribound}
 |y_i(t)| \leq C(t) (1+|y_0(\lambda_i)|),\quad i \in \{1,2\},
\end{equation}
and
\begin{equation}\label{eq:yLipschitz}
 |y_1(t)-y_2(t))| \leq C(t) |y_0(\lambda_1) - y_0(\lambda_2)|
\end{equation}
with $C(t)=\gamma\exp\left(w_0 (t-t_0)+\gamma\int_{t_0}^{t}k(s)\,ds\right)$.
\end{lemma}

For the cost function and control constraints, we will impose the following assumptions.
\begin{assumptions}\label{ass:CostAndConstraints}
The function $\varphi\: \Lambda \times C([t_0,\tf];Y) \times \Ut \times \Vt \to \RR$ is continuous and, for every $v \in \Vt$, the
functions $g^v_1,\ldots,g^v_M\: \Lambda \times \Ut \times [t_0,\tf] \to \RR$ are such that the set of admissible controls $\Wt(\lambda)$ is not empty for all 
$\lambda \in \Lambda$.
\end{assumptions}
In particular, under Assumptions~\ref{ass:ControlSys} and~\ref{ass:CostAndConstraints}, for every $\lambda \in \Lambda$,
the set $\Sscr_{[t_0,\tf]}(\lambda)$ is non-empty. Moreover, one obtains local Lipschitz
continuity of the value function if the perturbation parameter $\lambda$ acts Lipschitz continuously on $\varphi$ and $y_0$
by similar arguments as in a classical Banach or Hilbert space case \cite{CannarsaFrankowska1992,BarbuDaPrato1983}.

\begin{theorem}\label{thm:LipInitial} Under the Assumptions~\ref{ass:ControlSys} and~\ref{ass:CostAndConstraints}, suppose that
the constraint functions $g^v_1,\ldots,g^v_M$ are independent of $\lambda$. Let $\bar\lambda$ be some fixed parameter in $\Lambda$ and assume that for some
bounded neighborhood $B(\bar\lambda)$ of $\bar{\lambda}$ and some constant $L_0$
\begin{equation}\label{eq:y0Lip}
|y_0(\lambda_1)-y_0(\lambda_2)| \leq L_{0} \, |\lambda_1-  \lambda_2|,\,~\lambda_1, \lambda_2 \in B(\bar\lambda).
\end{equation}
Moreover, let $K=\sup_{\lambda \in B(\bar\lambda)}|y_0(\lambda)|$ and assume that for some constant $L_{\varphi}$
  \begin{equation}\label{eq:varphiLip}
   |\varphi(\lambda_1,y,u,v)-\varphi(\lambda_2,\bar y,u,v)| \leq L_{\varphi}(|y-\bar{y}|+|\lambda_1-\lambda_2|)
  \end{equation}
for all $(u,v)\in \Wt$, $y,\bar{y}$ such that $\max\{|y|,|\bar y|\} \leq C(\tf)(1+K)$ and
$\lambda_1$, $\lambda_2 \in B(\bar\lambda)$, where $C(t)$ is the
bound from Lemma~\ref{lem:bounds}. Then there exists a constant $\hat L_\nu$ such that
\begin{equation}\label{eq:nuLipInitial}
 |\nu(\lambda_1)-\nu({\lambda_2})| \leq \hat L_\nu  |\lambda_1 - \lambda_2|,\quad~\lambda_1, \lambda_2 \in B(\bar\lambda).
\end{equation}
\end{theorem}
\begin{proof}
Let $\varepsilon>0$ and $\lambda_1$, $\lambda_2 \in B(\bar\lambda)$ be given. Choose $(u_\varepsilon,v_\varepsilon) \in \Wt$ such that
\begin{equation*}
 \varphi(\lambda_2,\bar{y}_\varepsilon,u_\varepsilon,v_\varepsilon) \leq \nu(\lambda_2)+\varepsilon,
\end{equation*}
where $\bar{y}_\varepsilon$ denotes the reference solution $y(\cdot;y_0(\lambda_2),u_\varepsilon,v_\varepsilon)\in\Sscr_{[t_0,\tf]}$.
Let $y_\varepsilon$ denote the perturbed solution $y(\cdot;y_0(\lambda_1),u_\varepsilon,v_\varepsilon)\in\Sscr_{[t_0,\tf]}$.
Lemma~\ref{lem:bounds} and the assumptions yield
\begin{equation*}
 |y_\varepsilon(t)| \leq C(t) (1+K),~t \in [t_0,\tf],
\end{equation*}
and
\begin{equation*}
 |y_\varepsilon(t)- \bar{y}_\varepsilon(t)| \leq C(t) L_{0} |\lambda_1 -  \lambda_2|,~t \in [t_0,\tf].
\end{equation*}
Hence,
\begin{alignat*}{1}
 \varphi(\lambda_1,y_\varepsilon,u_\varepsilon,v_\varepsilon)
 &\leq \varphi(\lambda_2,\bar y_\varepsilon,u_\varepsilon,v_\varepsilon)+|\varphi(\lambda_1,y_\varepsilon,u_\varepsilon,v_\varepsilon) - \varphi(\lambda_2,\bar y_\varepsilon,u_\varepsilon,v_\varepsilon)|\\
 &\leq \varphi(\lambda_2,\bar y_\varepsilon,u_\varepsilon,v_\varepsilon)+L_\varphi (C(\tf) L_{0}+1) |\lambda_1 -  \lambda_2|.
\end{alignat*}
Thus
\begin{equation*}
\begin{aligned}
 \nu(\lambda_1) &\leq \varphi(\lambda_1,y_\varepsilon,u_\varepsilon,v_\varepsilon)
 \leq \varphi(\lambda_2,\bar y_\varepsilon,u_\varepsilon,v_\varepsilon)
 +L_\varphi (C(\tf) L_{0}+1) |\lambda_1 -  \lambda_2|
 \\
 &\leq  \nu({\lambda_2})+\varepsilon+L_\varphi (C(\tf) L_{0}+1) |\lambda_1 -  \lambda_2|.
\end{aligned}
\end{equation*}
Letting $\varepsilon \to 0$ from above gives an upper bound $\nu(\lambda_1) \leq \nu({\lambda_2})+\hat L_\nu |\lambda_1 -  \lambda_2|$
with 
\begin{equation}\label{eq:hatLnudef}
\hat L_\nu=L_\varphi (C(\tf) L_{0}+1). 
\end{equation}
Interchanging the roles of $\lambda_1$ and $\lambda_2$ yields the claim.
\end{proof}

In the subsequent section, we will obtain a similar result concerning the perturbation of the functions $g^v_1,\ldots,g^v_M$ and the cost function $\varphi$ under additional structural hypothesis and a constraint qualification.

\section{Perturbation of the constraints for convex problems}\label{sec:constraints}
In this section, we show that under a Slater-type condition the optimal value $\nu(\lambda)$ of the mixed-integer optimal control problem \eqref{eq:miocp}
in the case of a convex cost function and linear dynamics is locally Lipschitz continuous as a function of a parameter $\lambda$
acting on the control constraints $g^v_1,\ldots,g^v_M$ and the cost function $\varphi$. We need the following
\begin{assumptions}\label{ass:Convex}
The map $(y,u) \mapsto f(t,y,u,v)$ is linear and the map $(y,u) \mapsto \varphi(\lambda,y,u,v)$ is convex.
Moreover, the function $\varphi$ is Lipschitz continuous
with respect to $\lambda$ in the sense that
\begin{equation}
|\varphi(\lambda_1,\, y,\, u,\, v) - \varphi(\lambda_2,\, y,\, u,\, v)|
\leq
L_\varphi(|y|,\, |u|) \,|\lambda_1 - \lambda_2|
\end{equation}
with a continuous function $L_\varphi\: [0,\infty)^2 \rightarrow [0,\,\infty)$. For all $k=1,\ldots,M$, the maps $u \mapsto g_k^v(\lambda,\,u,\, t)$ are convex,
the maps $(u,t)\mapsto g_k^v(\lambda,\,u,\, t)$ are continuous and the functions $g^v_k$ are Lipschitz continuous 
with respect to $\lambda$ in the sense that for all $t \in [t_0,\tf]$
\begin{equation}
|g_k^v(\lambda_1,\, u,\, t) - g_k^v(\lambda_2,\, u,\, t)|
\leq
L_g(|u|) \,|\lambda_1 - \lambda_2|
\end{equation}
with a continuous function $L_g\: [0,\infty) \rightarrow [0,\,\infty)$.
\end{assumptions}

Under the Assumptions~\ref{ass:ControlSys}--\ref{ass:Convex} and assuming that $y_0$ is independent of $\lambda$,
we have for each parameter $\lambda \in \Lambda$ the mixed-integer optimal control problem \eqref{eq:miocp} with 
$y_0(\lambda)$ replaced by a fixed initial state $y_0 \in Y$. Moreover, in this section, $\nu(\lambda)$ denotes
the corresponding optimal value function \eqref{eq:defnu} with fixed initial state $y_0$.
The subsequent analysis is based upon the presentation in \cite{gu97}, where for the finite
dimensional case the existence of the one sided derivatives of the optimal value function $\nu(\lambda)$
is shown. For a generalization to the above setting, we first introduce a Slater-type constraint qualification,
a dual problem and prove a strong duality result.

\begin{assumptions}{\bf (CQ)}\label{ass:CQ}
For some $\bar \lambda \in \Lambda$ and some bounded neighborhood $B(\bar \lambda)\subset \Lambda$ of $\bar \lambda$
there exists a number $\omega>0$ such that for all $v\in \Vt$ there is a Slater point $\bar u_v \in U$ such that
for all $\lambda \in B(\bar \lambda)$ we have
\begin{equation}\label{eq:slaterpoint}
g_k^v(\lambda, \bar u_v,t) \leq - \omega\quad\mbox{\rm for all } t \in [t_0,\tf],~k=1,\ldots,M,
\end{equation}
\begin{equation}
\label{30}
\sup_{v\in V_{[t_0,\,\tf]}}  \sup_{\lambda \in B(\bar \lambda)} \varphi(\lambda,y(\bar u_v,v),\bar u_v,v)<\infty
\end{equation}
and that there exists a number $\underline \alpha$ such that for all $\lambda \in B(\bar \lambda)$ we have
\begin{equation}
\label{31}
\nu(\lambda) \geq \underline \alpha
\end{equation}
and that the set
\begin{equation}
\label{coercive}
\hspace*{-1em}
\begin{aligned}
\bigcup_{\lambda_1,\lambda_2 \in B(\bar \lambda)}
\biggl\{& (u,v) \in U_{[t_0,\tf]}\times \Vt:\\
&~\varphi(\lambda_1,y(u,v),u,v) \leq \varphi(\lambda_1,y(\bar u_v,v),\bar u_v,v) + |\lambda_1 - \lambda_2|^2,\\
&~g_k^v(\lambda_1,u(t),t) \leq 0,~t \in [t_0,\tf],~k=1,\ldots,M\biggr\}=:\bar S(y_0)
\end{aligned}
\end{equation}
is bounded.
\end{assumptions}

Note that (\ref{31}) holds if
$\underline \alpha$ is a lower bound for the cost function.

Let $C([t_0,\tf])^\ast_+$ denote the set of  positive function of bounded variation on $[t_0,\tf]$.
For any controls $v \in \Vt$, $u \in \Ut$ and any
$\mu^\ast \in \left(C([t_0,\tf])^\ast_+\right)^M$
we define the Lagrangian
\begin{equation}\label{lagrangian}
\Lagr_v(\lambda,u,\mu^\ast) = \varphi(\lambda,\,y(u,v),\,u,\, v) + \sum_{k=1}^M \int_{t_0}^{\tf}  g_k^v(\lambda,u,s)\,\mathrm{d} \mu^\ast_k(s),
\end{equation}
where the integral is in the Riemann-Stieltjes sense. Further, we define
\begin{equation} \label{lagrangian1}
h_v(\lambda,\mu^\ast) = \inf_{u \in \Ut} \Lagr_v(\lambda,u,\mu^\ast).
\end{equation}

Under the constraint qualification \CQref, for all fixed $\lambda \in B(\bar \lambda)$ and $v\in \Vt$,
the classical convex duality theory as presented in \cite{ekturn} implies
the strong duality result (see also \cite{gu:onesi}) %, \cite{shapiro})
\begin{equation}\label{duality}
\sup_{\mu^\ast \in \left(C([t_0,\tf])^\ast_+\right)^M} h_v(\lambda,\mu^\ast) = \nu^v(\lambda)
\end{equation}
where $\nu^v(\lambda)$ denotes the optimal value of the following convex optimal control problem
only in the variables $y$ and $u$
\begin{equation}\label{eq:primalvfix}
\left.
\begin{array}{l}
 \text{minimize}~\varphi(\lambda,y,u,v)~\text{subject to}\\
 \quad \dot{y}(t)  =    A \, y(t)  + f(t,\,y(t),\, u(t), \, v(t)),\; t \in (t_0,\, \tf) \;{\rm a.e.},
 %+ B u(t) + F(t,v(t)),
 ~y(0)= y_0,
 \\
 \quad g_k^v(\lambda,u,t) \leq 0~\text{for all}~t \in [t_0,\tf],~k \in \{1,\ldots,M\},\\
 \quad y \in C([t_0,\tf];Y),~u \in \Ut,
\end{array}
\right\}
\end{equation}
see, for example, \cite{gu:onesi,ekturn}. Further, we introduce the sets
\begin{equation}
  F_v(\lambda) =\{\mu^\ast \in \left(C([t_0,\tf])^\ast_+\right)^M : h_v(\lambda,\mu^\ast) > - \infty\}
\end{equation}
and
\begin{equation}
G(\lambda) = \left\{\rho \in \prod_{v \in \Vt} F_v(\lambda) : \inf_{v\in \Vt} h_v(\lambda,\rho_v) \in \RR\right\},
\end{equation}
and, for $(r,\rho) \in \mathbb{R}\times G(\lambda)$, we define the projection $\pi(r,\rho)=r$. Finally, we introduce the
following maximization problem as the dual problem of \eqref{eq:miocp} %}\eqref{eq:primalconvex}
\begin{equation}\label{eq:dualconvex}
\left.
 \begin{array}{l}
 \text{maximize}~\pi(r,\rho)~\text{subject to}\\
 \quad \rho \in G(\lambda),~r \in  \mathbb{R},\\
 \quad r \leq h_v(\lambda, \rho_v)~\text{for all}~v\in \Vt.
 \end{array}\right\}
\end{equation}
The optimal value of this dual problem is
\begin{equation}
\Delta(\lambda) = \sup_{\rho\in G(\lambda)} \inf_{v\in \Vt} h_v(\lambda,\rho_v).
\end{equation}
Now we state a strong duality result.
For the convenience of the reader we also present a complete proof.
Note however that Theorem \ref{strongduality} can also be deduced
from Ky Fan's minimax theorem in
\cite{borweinzhuang86}.
\begin{theorem}[Strong duality]
\label{strongduality}
The constraint qualification \CQref~implies that
\begin{equation}
\nu(\lambda) = \Delta(\lambda),~\text{for all}~\lambda \in B(\bar \lambda),
\end{equation}
where $\nu(\lambda)$ is the optimal value of \eqref{eq:miocp} with fixed initial state.
\end{theorem}
\begin{proof}
Choose $\rho \in G(\lambda)$. Then convex weak duality implies that for all $v\in \Vt$ we have
\begin{equation}
h_v(\lambda,\rho_v) \leq \nu^v(\lambda).
\end{equation}
Thus
\begin{equation}
\inf_{v\in \Vt} h_v(\lambda,\rho_v) \leq \inf_{v\in \Vt} \nu^v(\lambda)=\nu(\lambda).
\end{equation}
This implies that
\begin{equation}
\Delta(\lambda) = \sup_{\rho\in G(\lambda)} \inf_{v\in \Vt} h_v(\lambda,\rho_v) \leq \nu(\lambda),
\end{equation}
that is, we have shown the weak duality. Further, due to \CQref  $\,$ and convex strong duality
from \eqref{duality}, for each $v \in \Vt$, we can choose some $\mu^\ast_v \in \left(C([t_0,\tf])^\ast_+\right)^M$
such that
\begin{equation}
h_v(\lambda, \mu^\ast_v) = \nu^v(\lambda).
\end{equation}
Define $\rho^\ast = (\mu^\ast_v)_{v\in \Vt}$. Then $\rho^\ast \in G(\lambda)$.
This yields
\begin{equation}
\begin{aligned}
\Delta(\lambda) &=\sup_{\rho\in G(\lambda)} \inf_{v\in \Vt} h_v(\lambda,\rho_v)\\
&\geq \inf_{v \in \Vt} h_v(\lambda,\mu^\ast_v)=\inf_{v \in \Vt} \nu^v(\lambda) = \nu(\lambda).
\end{aligned}
\end{equation}
Hence the strong duality follows.
\end{proof}

Based upon the above duality concept, we can now show the
Lipschitz continuity
of the optimal value function 
in a neighborhood of $\bar\lambda$. 
To this end, we introduce
for any $\varepsilon \geq 0$ the set of $\varepsilon$-optimal points
\begin{equation}
 \begin{aligned}
  P(\lambda,\varepsilon) =~&\bigl\{ u \in U_{[t_0,\tf]} : \text{there exists}~v \in \Vt~\text{such that}\\
  &\quad g_k^v(\lambda,u,t) \leq 0~\text{for all}~t \in [t_0,\tf],~k=1,\ldots,M,\\
  &\quad \varphi(\lambda,\,y(u,v),\,u,\, v) \leq \nu(\lambda) + \varepsilon \bigr\}
 \end{aligned}
\end{equation}
and we set $H(\lambda,\varepsilon) = \{ \rho \in G(\lambda) : \inf_{v\in V} h_v(\lambda,\rho_v) \geq  \nu(\lambda)- \varepsilon\}$.

\begin{lemma}\label{bounded}
Under \CQref, the set
\begin{equation}
\Omega(\bar \lambda   ):=\bigcup_{\lambda_1, \lambda_2 \in B(\bar \lambda),\,v\in \Vt} \left\{\rho_v : \rho \in H(\lambda_1,|\lambda_1 -  \lambda_2|^2) \right\}
\end{equation}
is bounded.
\end{lemma}
\begin{proof} Due to assumption \CQref, for all $v \in \Vt$, we have the Slater point $\bar u_v$.
Choose  $\lambda_1$, $\lambda_2 \in B(\bar \lambda)$ 
and
$\rho \in H(\lambda_1,|\lambda_1 - \lambda_2|^2)$. Then $\inf_{v\in \Vt} h_v(\lambda_1,\rho_v) \geq \nu(\lambda_1)- |\lambda_1 - \lambda_2|^2$.
Thus by definition of $h_v$, for all $v\in \Vt$, we have that $\Lagr_v(\lambda_1,\bar u_v,\rho_v) \geq h_v(\lambda_1,\rho_v) \geq \nu(\lambda_1) - |\lambda_1 - \lambda_2|^2$.
By definition of $\Lagr_v$, this implies
\begin{equation}
\varphi(\lambda_1,y(\bar u_v,v),\bar u_v,v) + \sum_{k=1}^M \int_{t_0}^{\tf} g_k^v(\lambda_1,\bar u_v,s)\,\mathrm{d}(\rho_v)_k(s) \geq \nu(\lambda_1) - |\lambda_1 - \lambda_2|^2.
\end{equation}
Now using that
\begin{equation}
g_k^v(\lambda_1,\bar u_v,t) \leq - \omega<0~\text{for all}~t \in [t_0,\tf],~k \in \{1,\ldots,M\},
\end{equation}
we can divide by $-\omega<0$ and obtain due to (\ref{31})
\begin{eqnarray*}
\sum_{k=1}^M \int_{t_0}^{\tf} 1\,\mathrm{d} (\rho_v)_k(s) & \leq & \frac{ \nu(\lambda_1) - |\lambda_1 - \lambda_2|^2 - \varphi(\lambda_1,y(\bar u_v,v),\bar u_v, v)}{-\omega}\\
& = & \frac{|\lambda_1 - \lambda_2|^2 + \varphi(\lambda_1,y(\bar u_v,v),\bar u_v,v) - \nu(\lambda_1)}{\omega}\\
& \leq & \frac{|\lambda_1 - \lambda_2|^2 + \varphi(\lambda_1,y(\bar u_v,v),\bar u_v,v) - \underline \alpha}{\omega}\\
& \leq & \frac{|\lambda_1 - \lambda_2|^2 + \sup\limits_{v\in \Vt} \sup\limits_{\lambda \in B(\bar \lambda)} \varphi(\lambda,y(\bar u_v,v),\bar u_v,v) - \underline \alpha}{\omega}.
\end{eqnarray*}
Due to (\ref{30}) this yields the assertion.
\end{proof}

\begin{lemma} \label{liminf}
Suppose that \CQref~holds.
Then for all $\lambda_1$, $\lambda_2 \in B(\bar \lambda)$
we have
\begin{equation}
\nu(\lambda_1) - \nu(\lambda_2)
\geq
 -\underline C \, |\lambda_1 - \lambda_2|
\end{equation}
for some $\underline C$ in $\RR$.
\end{lemma}
\begin{proof}

Let
$\lambda_1$, $\lambda_2 \in  B(\bar\lambda)$ be given.
Choose a solution $u\in P(\lambda_1,|\lambda_1-  \lambda_2|^2)$
and $\tilde v \in \Vt$ with
$g_j^{\tilde v}(\lambda_1, u, t) \leq 0$ for all $t \in [t_0,\tf]$, $j=1, \ldots,M$,
$\varphi(\lambda_1, \, y(u,\tilde v),\,u,\, \tilde v)
\leq \nu(\lambda_1) + |\lambda_1- \lambda_2|^2$ and $
\bar \rho \in H( \lambda_2,  |\lambda_1-   \lambda_2|^2)$.
Then we have
\begin{eqnarray*}
\nu(\lambda_1) - \nu(\lambda_2)
& \geq &
   \;\varphi(\lambda_1, y(u,\tilde v),u,\tilde v) - \inf_{v\in \Vt} h_v( \lambda_2, \bar \rho_v) - 2 |\lambda_1 - \lambda_2|^2
  \\
  & \geq &
 \;\varphi(\lambda_1, y(u,\, \tilde v),u, \, \tilde v) - h_{\tilde  v}( \lambda_2, \bar \rho_{\tilde v}) - 2 |\lambda_1 -  \lambda_2|^2
  \\
  & \geq &
  \;\varphi(\lambda_1, y(u,\, \tilde v),\, u, \tilde v) - \Lagr_{\tilde v}(\lambda_2,  u, \bar \rho_{\tilde v}) - 2 |\lambda_1 - \lambda_2|^2
 \\
   & \geq &
  \;\varphi(\lambda_1, y(u, \tilde v),u, \tilde  v)
  + \sum_{j=1}^M \int_{t_0}^{\tf}  g_j^{\tilde  v}(\lambda_1,u,s)\,\mathrm{d} \bar \rho_{\tilde v}(s)
 \\
  & ~ &
  \qquad~-\Lagr_{\tilde v}(\lambda_2, u, \bar \rho_{\tilde v}) - 2 |\lambda_1 - \lambda_2|^2
  \\
  & = &
  \Lagr_{\tilde v} (\lambda_1,\, u, \, \bar\rho_{\tilde v}) - \Lagr_{\tilde v} (\lambda_2,\, u, \, \bar\rho_{\tilde v}) - 2 |\lambda_1-  \lambda_2|^2
  \\
  & \geq &
  - \biggl[ L_\varphi( |y(u,\,\tilde  v)|,\, |u|)\\
  & ~ & \qquad~+  M \, L_g(|u|) \, \int_{t_0}^{\tf} \, d \bar \rho_{\tilde v}(s) + 2 |\lambda_1 -  \lambda_2| \biggr]\, |\lambda_1 -  \lambda_2 |.
\end{eqnarray*}
Due to (CQ), the set $\bar S(y_0)$ from (\ref{coercive}) is bounded. Thus our assumptions imply that the set
$\{y(\hat u,\, \hat v):\, (\hat u,\,\hat v)\in \bar S(y_0)\}$
is bounded (see (\ref{eq:aprioribound})).
Due to Lemma \ref{bounded}, the set $\Omega(\bar\lambda)$
is also bounded.
Since $L_\varphi$ and $L_g$ are continuous this allows us
to define the real  number
\begin{equation}
\label{underlinecdefinition}
\begin{split}
\tilde C =
&\sup_{(\hat u,\,\hat v)\in \bar S(y_0)}\, L_\varphi( |y(\hat u,\, \hat v)|,\, |\hat u|)\\
&\quad+  M\, L_g(|\hat u|) \,
 \sup_{ \hat \rho_w \in  \Omega(\bar \lambda)}  \int_{t_0}^{\tf} \, d \hat \rho_{w}(s) \;
  +2 \sup_{\lambda_1,\, \lambda_2 \in B(\bar \lambda)}|\lambda_1-\lambda_2|.
\end{split}
\end{equation}
Due to the definition of $ P(\lambda_1,\, |\lambda_1 - \lambda_2|^2)$
we have
$(u,\, \tilde v) \in \bar S(y_0)$.
Moreover, we have
$\bar \rho_{\tilde v} \in \Omega (\bar\lambda)$.
Hence we have
\begin{eqnarray*}
\nu(\lambda_1) - \nu(\lambda_2)
& \geq &
- \tilde C\; |\lambda_1 -  \lambda_2|
  \end{eqnarray*}
  and the assertion follows with
$  \underline C =  \tilde C$. 
\end{proof}

Similarly as in Lemma \ref{liminf},
by interchanging the roles of $\lambda_1 $ and $\lambda_2$,
and
with the choice
$\overline C =\tilde C$ with $\tilde C$ as defined in (\ref{underlinecdefinition})
we can prove the following Lemma:
\begin{lemma}\label{limsup}
Suppose that \CQref~holds.
Then, for all $\lambda_1,\, \lambda_2 \in B(\bar \lambda)$,
we have
\begin{equation}
\nu(\lambda_1) - \nu(\lambda_2)
\leq
 \overline C \, |\lambda_1 - \lambda_2|,
\end{equation}
for some $\overline C$ in $\RR$.
\end{lemma}
The above analysis implies our main result about the Lipschitz continuity of the optimal value as a function of the parameter $\lambda$.
\begin{theorem}\label{thm:LipConstraints}
Under the Assumptions~\ref{ass:ControlSys}--\ref{ass:Convex}, for any $\bar\lambda \in \Lambda$ and a bounded neighborhood $B(\bar\lambda) \subset \Lambda$ satisfying the constraint qualification \CQref~it holds %for all
\begin{equation}\label{eq:nuLipCostConstraint}
 |\nu(\lambda_1) - \nu(\lambda_2)|
\leq
\tilde C
\,
|\lambda_1 -  \lambda_2|\quad\text{for all}~\lambda_1,\,\lambda_2 \in B(\bar \lambda)
\end{equation}
with $\tilde C$ as defined in \eqref{underlinecdefinition},
that is, 
the optimal value function $\nu$ is Lip\-schitz continuous 
in a neighborhood of $\bar \lambda$ 
with Lipschitz constant $\tilde C$.
\end{theorem}
\begin{proof}
The result follows from combining the proofs of Lemma~\ref{liminf} and~\ref{limsup}.
\end{proof}

\section{Joint perturbations}\label{sec:joint}
In this section, we study the joint local Lipschitz continuity of the value function $\nu$ with respect to $\lambda$ 
acting on the initial data, the constraints and the costs. We consider the mixed-integer optimal control problem \eqref{eq:miocp}.
In contrast to Section~\ref{sec:constraints} the initial state $y_0(\lambda)$ depends on $\lambda$. Also, the constraints
and the objective function depend on $\lambda$. The result is obtained by combining Theorem~\ref{thm:LipInitial} and~\ref{thm:LipConstraints}.

\begin{theorem}\label{thm:JointLip} Under the Assumptions~\ref{ass:ControlSys}--\ref{ass:Convex}, for any $\bar\lambda \in \Lambda$, a bounded neighborhood $B(\bar\lambda) \subset \Lambda$ 
let $L_0,L_{\varphi}$ be constants such that \eqref{eq:y0Lip} and \eqref{eq:varphiLip} hold as in Theorem~\ref{thm:LipInitial}. Further, suppose that (CQ) holds in the sense that \eqref{eq:slaterpoint} is satisfied and $\cup_{y_0 \in Y_0} \bar S(y_0)$ is bounded with $\bar S(y_0)$ from \eqref{coercive} and $Y_0 = \{y_0(\lambda) : \lambda \in B(\bar\lambda)\}$.
Then, 
there exists a constant $L_\nu$ such that
\begin{equation}\label{eq:nuJointlyLip}
|\nu(\lambda_1)-\nu(\lambda_2)| \leq L_\nu |\lambda_1 - \lambda_2|,\quad\text{for all}~\lambda_1,\,\lambda_2 \in B(\bar \lambda),
\end{equation}
where $\nu(\lambda)$ is the optimal value of \eqref{eq:miocp} as defined in \eqref{eq:defnu}.
\end{theorem}
\begin{proof} In this proof, for $\lambda \in B(\bar\lambda)$ and $y_0 \in Y$, we use the notation
\begin{equation}\label{eq:defnu2}
\begin{array}{l}
 \nu(\lambda,y_0)=\inf \bigl\{\varphi(\lambda,y,u,v) :\\
 \quad \dot{y}(t)  =   A \, y(t)  + f(t,\,y(t),\, u(t), \, v(t)),\; t \in (t_0,\, \tf) \;{\rm a.e.},
 %+ B u(t) + F(t,v(t)),
 ~y(0)= y_0,
 %~t \in (t_0,\,\tf),
 \\
 \quad g_k^v(\lambda,u,t)  \leq 0~\text{for all}~t \in [t_0,\tf],~k=1,\ldots,M,\\
 \quad y \in C([t_0,\tf];Y),~u \in \Ut,~v \in \Vt\bigr\}.
\end{array}
\end{equation}
Due to \eqref{eq:y0Lip} and \eqref{eq:varphiLip} the set $Y_0$ is bounded by the constant $K$ from Theorem~\ref{thm:LipInitial} 
and for all $y_0 \in Y_0$, $v \in \Vt$, $\lambda \in B(\bar\lambda)$ we have the upper bound
\begin{equation*}
\begin{split}
 \varphi(\lambda,y(y_0,\bar u_v,v),\bar u_v,v) \leq \varphi(\bar\lambda,&\,y(y_0(\bar\lambda),\bar u_v,v),\bar u_v,v)\\
 +~L_{\varphi}(|y(y_0,\bar u_v,v),\bar u_v,v)~-&~y(y_0(\bar\lambda),\bar u_v,v),\bar u_v,v)|+|\lambda-\bar\lambda|).
 \end{split}
\end{equation*}
Moreover, from Lemma~\ref{lem:bounds}, we obtain $|y(y_0,\bar u_v,v)| \leq C(\tf)(1+K)$. This implies \eqref{30}. Using similar arguments, we get a lower bound
\begin{equation*}
\nu(\lambda,y_0)\geq\inf_{y_0 \in Y_0} \inf_{v\in V_{[t_0,\,\tf]}}  \inf_{\lambda \in B(\bar \lambda)} \varphi(\lambda,y(y_0,\bar u_v,v),\bar u_v,v)=:\underline\alpha>-\infty.
\end{equation*}
This implies \eqref{31} with $\underline\alpha$ independent of $\lambda$. Thus Assumptions~\ref{ass:CQ} holds for all $y_0 \in Y_0$ and the proof of Lemma~\ref{bounded} shows that
the bound of the set $\Omega(\bar\lambda)$ is independent of $y_0$.
Due to \eqref{eq:y0Lip} the function $L_{\varphi}$ in Assumptions~\ref{ass:Convex} is constant and $\tilde C$ from \eqref{underlinecdefinition} reduces to 
\begin{equation*}%\label{underlinecdefinition2}
\begin{split}
\tilde C = L_\varphi~+  &~M \sup_{y_0 \in Y_0} \sup_{(\hat u,\hat v)\in \bar S(y_0)} L_g(|\hat u|) \, \sup_{ \hat \rho_w \in  \Omega(\bar \lambda)}  \int_{t_0}^{\tf} \, d \hat \rho_{w}(s) \\
&+ 2 \sup_{\lambda_1,\, \lambda_2 \in B(\bar \lambda)}|\lambda_1-\lambda_2|<\infty.
\end{split}
\end{equation*}
Now, let $\lambda_1,\lambda_2 \in B(\bar{\lambda})$. From Theorem~\ref{thm:LipInitial} with $\lambda_1$ as first argument of $\nu$ fixed we get
$|\nu(\lambda_1,y_0(\lambda_1))-\nu(\lambda_1,y_0(\lambda_2))|\leq \hat L_{\nu} |\lambda_1-\lambda_2|$ with $\hat L_{\nu}$ given by \eqref{eq:hatLnudef}. 
From Theorem~\ref{thm:LipConstraints} with $\lambda_2$ as an argument of $y_0$ fixed we get
$|\nu(\lambda_1,y_0(\lambda_2))-\nu(\lambda_2,y_0(\lambda_2))|\leq \tilde C |\lambda_1-\lambda_2|$. Thus we obtain the inequality
\begin{equation*}
 \begin{aligned}
  & |\nu(\lambda_1,y_0(\lambda_1))-\nu(\lambda_2,y_0(\lambda_2))|\\
  & \quad \leq  |\nu(\lambda_1,y_0(\lambda_1))-\nu(\lambda_1,y_0(\lambda_2))| + |\nu(\lambda_1,y_0(\lambda_2))-\nu(\lambda_2,y_0(\lambda_2))|\\
  & \quad \leq (\hat L_{\nu}+\tilde C) |\lambda_1-\lambda_2|\\
 \end{aligned}
 \end{equation*}
and \eqref{eq:nuJointlyLip} follows with $L_{\nu}=\hat L_{\nu}+\tilde C$. 
\end{proof}

\section{Example}\label{sec:example}
We discuss an academic application concerning the optimal positioning of an actuator motivated from applications in
thermal manufacturing \cite{IftimeDemetriou2009,HanteSager2013}.

\begin{example} Suppose that $\Omega \subset \RR^2$ is a bounded domain containing two non-overlapping control
domains $\omega_1$ and $\omega_2$. For simplicity, we assume that the boundaries of all these domains $\partial \Omega$, $\partial \omega_1$
and $\partial \omega_2$ are smooth. Let $\varepsilon,\delta>0$ be two given parameters, let $\chi_{\omega_i}$ denote the
characteristic function of $\omega_i$, let $v|_{[t_1,t_2]^+}$ denote the restriction
of $v$ to the non-negative part of the interval $[t_1,t_2]$ and let $\Delta y$ denote the Laplace operator.
For a time horizon with $\tf-t_0 > \delta$, we consider
the optimal control problem
\begin{equation*}
\begin{aligned}
&\text{minimize} ~\int_{t_0}^{\tf} \int_\Omega |y(t,x)-\hat{y}(t,x)|^2\,dx\,dt + \int_{t_0}^{\tf} |u(t)|^2\,dt\\
 &y_t - \Delta y + v(t)u(t) \chi_{\omega_1} + (1-v(t))u(t) \chi_{\omega_2} = 0\quad \text{on}~(t_0,\tf) \times \Omega\\
 &y=0\quad \text{on}~(t_0,\tf) \times \partial\Omega\\
 &y=\bar y\quad \text{on}~\{t_0\} \times \Omega\\
 &v(t) \in \Vcal=\{0,1\}\quad \text{on}~[t_0,\tf]\\
 &u(t) \in \begin{cases}[0,1+\varepsilon] & ~\text{if}~v|_{[t-\delta,t]^+} \equiv 1~\text{or}~v|_{[t-\delta,t]^+} \equiv 0 ~\text{a.\,e. on}~[t-\delta,t]^+\\
               [0,\varepsilon] &~\text{else}.
              \end{cases}
\end{aligned}
\end{equation*}
The combination of the actuator and constraints in this problem model that the continuous control $u$ is 
%switched off
restricted  to a small uncontrollable disturbance $\varepsilon$
for a dwell-time period of length $\delta$ whenever a decision was taken to change the control region $\omega_1$ to $\omega_2$ or vice versa while the goal is
to steer the initial state $\bar{y} \in L^2(\Omega)$ as close as possible to a desired state $\hat{y} \in C([t_0,\tf];L^2(\Omega))$.
We consider a perturbation $\lambda=(\bar{y},\varepsilon,\hat{y})$, i.\,e., a joint perturbation of initial data, the disturbance and the tracking target.

We can consider this problem in the abstract setting with $Y=L^2(\Omega)$, $U=\RR$, $\Ut=L^\infty(t_0,\tf)$
\begin{equation*}
\begin{aligned}
&Ay = \laplace y,~y \in D(A)=H^2(\Omega) \cap H^{1}_0(\Omega),\\
&f(y,u,v)=f(u,v)=-(vu \chi_{\omega_1} + (1-v)u \chi_{\omega_2}),\\
&\varphi(\lambda,y,u,v)=\int_{t_0}^{\tf} \int_\Omega |y(t,x)-\hat{y}(t,x)|^2\,dx\,dt + \int_{t_0}^{\tf} |u(t)|^2\,dt,
\end{aligned}
\end{equation*}
defining
\[\bar u^v_\varepsilon(t) = \begin{cases}1+\varepsilon & ~\text{if}~v|_{[t-\delta,t]^+} \equiv 1~\text{or}~v|_{[t-\delta,t]^+} \equiv 0 ~\text{a.\,e. on}~[t-\delta,t]^+\\
               \varepsilon &~\text{else},
              \end{cases}\]
and setting $M=2$ and, for all $v \in \Vt$
\begin{eqnarray*}
g^v_1(\lambda,\, u,\, t) & = & \displaystyle \esssup_{s \in [t_0,\tf]} (u(s)-\bar u^v_\varepsilon(s)),
\\
g^v_2(\lambda,\, u,\, t) & = & \displaystyle \esssup\limits_{s \in [t_0,\tf]} (-u(s)).
\end{eqnarray*}
Here the $g^v_i(\lambda,\, u,\, \cdot)$ are constant with respect to $t$ and hence
continuous as functions of $t$.
Moreover, the maps $u \mapsto g^v_i(\lambda,\, u,\, \cdot)$ are
continuous in $L^\infty (t_0, \, \tf)$.
The objective function is convex with respect to $(y,u)$ 
and also the  maps $u \mapsto g^v_i(\lambda,\, u,\, t)$ are
convex.
Let $\varepsilon_1$, $\varepsilon_2>0$ be such that
without restriction we have
$\esssup_{s \in [t_0,\tf]} (u(s)-\bar u^v_{\varepsilon_1}(s))
\geq
\esssup_{s \in [t_0,\tf]} (u(s)-\bar u^v_{\varepsilon_2}(s))$.
Then we have
\begin{eqnarray*}
& &
|
g^v_1(\lambda_1,\, u,\, t)-
g^v_1(\lambda_2,\, u,\, t)
|
\\
&
=
&
\esssup_{s \in [t_0,\tf]} (u(s)-\bar u^v_{\varepsilon_1}(s))
-
\esssup_{s \in [t_0,\tf]} (u(s)-\bar u^v_{\varepsilon_2}(s))
\\
& =  &
\esssup_{s \in [t_0,\tf]} (u(s)+\bar u^v_{\varepsilon_2}(s) 
-\bar u^v_{\varepsilon_1}(s)-\bar u^v_{\varepsilon_2}(s))+\\
& & \qquad \qquad \qquad \qquad \qquad \qquad
-
\esssup_{s \in [t_0,\tf]} (u(s)-\bar u^v_{\varepsilon_2}(s))
\\
&
\leq
&
 \esssup_{s \in [t_0,\tf]}
 |\bar u^v_{\varepsilon_2}(s)   -\bar u^v_{\varepsilon_1}(s)|
 \\
&
\leq
&
|\varepsilon_2 - \varepsilon_1|
 .
\end{eqnarray*}
It is well-known that $(A,D(A))$ is the generator of a strongly continuous (analytic) semigroup of contractions $\{T(t)\}_{t \geq 0}$ on $Y$, see, e.\,g., \cite{Pazy1983}.
Also, Assumptions~\ref{ass:ControlSys}--\ref{ass:Convex} are easily verified and it is easy to see that \eqref{eq:y0Lip} and \eqref{eq:varphiLip} hold.
The constraint qualification (CQ) is satisfied with $\omega=\frac{\varepsilon}{2}$ and $\underline \alpha=0$.
%Note that the coerciveness of the objective function 
The control constraints imply that the set $\bar S$ is bounded in $\Ut \times \Vt$ independently of the initial state $y_0$.
Hence we can conclude from Theorem~\ref{thm:JointLip} that the optimal value function $\nu$ is locally Lipschitz continuous jointly
as a function of $\lambda=(\bar{y},\varepsilon,\hat{y})$.
\end{example}

\section{Conclusion}
We have studied the optimal value function for control problems on Banach spaces
that involve both continuous and discrete control decisions. For control systems of
a semilinear type subject to control constraints, we have shown that the optimal value
depends locally Lipschitz continuously on perturbations of the initial data and 
costs under natural assumptions. For problems consisting of linear systems on a Banach 
space subject to convex
control inequality constraints, we have shown that the optimal value of convex cost
functions depend locally Lipschitz continuously on Lipschitz continuous
perturbations of the costs and the constraints under a Slater-type constraint
qualification. The result has been obtained by proving a strong duality for an appropriate
dual problem. 

By a combination of the above results we have for the linear, convex case 
obtained local Lipschitz continuity jointly for parametric initial data, control 
constraints and cost functions. The Example~\ref{ex:nonsmooth} shows that this result
is sharp in the sense that we can, in general, not expect much more regularity 
than we have proved.

Our analysis currently does not address the stability of the optimal control under 
perturbations. This is an interesting direction for future work.

\section*{Acknowledgements}
This work was supported by the DFG grant CRC/Transregio 154, projects C03 and A03. The authors thank the reviewers 
and the editors from \emph{Math. Control Signals Syst.} for the constructive suggestions. The final 
publication is available at \href{http://link.springer.com/article/10.1007/s00498-016-0183-4}{link.springer.com}
and \href{http://dx.doi.org/10.1007/s00498-016-0183-4}{doi:10.1007/s00498-016-0183-4}.


\begin{thebibliography}{10}

\bibitem{BarbuDaPrato1983}
Viorel Barbu and Giuseppe Da~Prato.
\newblock {\em Hamilton-{J}acobi equations in {H}ilbert spaces}, volume~86 of
  {\em Research Notes in Mathematics}.
\newblock Pitman (Advanced Publishing Program), Boston, MA, 1983.

\bibitem{BensoussanDaPratoDelfourMitter1992}
Alain Bensoussan, Giuseppe Da~Prato, Michel~C. Delfour, and Sanjoy~K. Mitter.
\newblock {\em Representation and control of infinite-dimensional systems.
  {V}ol. 1}.
\newblock Systems \& Control: Foundations \& Applications. Birkh\"auser Boston
  Inc., Boston, MA, 1992.

\bibitem{bonnansshapiro}
J.~Fr{\'e}d{\'e}ric Bonnans and Alexander Shapiro.
\newblock {\em Perturbation analysis of optimization problems}.
\newblock Springer Series in Operations Research. Springer-Verlag, New York,
  2000.

\bibitem{borweinzhuang86}
J.~M. Borwein and D.~Zhuang.
\newblock On fan's minimax theorem.
\newblock {\em Math Progam}, 26:232--234, 1986.

\bibitem{CannarsaFrankowska1992}
Piermarco Cannarsa and Halina Frankowska.
\newblock Value function and optimality conditions for semilinear control
  problems.
\newblock {\em Appl Math Optim}, 26(2):139--169, 1992.

\bibitem{ekturn}
Ivar Ekeland and Thomas Turnbull.
\newblock {\em Infinite-dimensional optimization and convexity}.
\newblock Chicago Lectures in Mathematics. University of Chicago Press,
  Chicago, IL, 1983.

\bibitem{MR669727}
Jacques Gauvin and Fran{\c{c}}ois Dubeau.
\newblock Differential properties of the marginal function in mathematical
  programming.
\newblock {\em Math Programming Stud}, (19):101--119, 1982.

\bibitem{Gerdts2006}
Matthias Gerdts.
\newblock A variable time transformation method for mixed-integer optimal
  control problems.
\newblock {\em Optimal Control Appl. Methods}, 27(3):169--182, 2006.

\bibitem{gu:onesi}
M.~Gugat.
\newblock One-sided derivatives for the value function in convex parametric
  programming.
\newblock {\em Optimization}, 28(3-4):301--314, 1994.

\bibitem{gu97}
M.~Gugat.
\newblock Parametric disjunctive programming: one-sided differentiability of
  the value function.
\newblock {\em J Optim Theory Appl}, 92(2):285--310, 1997.

\bibitem{HanteSager2013}
Falk~M. Hante and Sebastian Sager.
\newblock Relaxation methods for mixed-integer optimal control of partial
  differential equations.
\newblock {\em Comput Optim Appl}, 55(1):197--225, 2013.

\bibitem{IftimeDemetriou2009}
Orest~V. Iftime and Michael~A. Demetriou.
\newblock Optimal control of switched distributed parameter systems with
  spatially scheduled actuators.
\newblock {\em Automatica J. IFAC}, 45(2):312--323, 2009.

\bibitem{MR2421286}
B.~S. Mordukhovich, N.~M. Nam, and N.~D. Yen.
\newblock Subgradients of marginal functions in parametric mathematical
  programming.
\newblock {\em Math Program}, 116(1-2, Ser. B):369--396, 2009.

\bibitem{Pazy1983}
A.~{Pazy}.
\newblock {\em Semigroups Of Linear Operators And Applications To Partial
  Differential Equations}.
\newblock Applied Mathematical Sciences Series, Springer-Verlag, New York,
  1983.

\bibitem{RuefflerHante2016}
Fabian Rüffler and Falk~M. Hante.
\newblock Optimal switching for hybrid semilinear evolutions.
\newblock {\em Nonlinear Anal Hybrid Syst}, 22:215--227, 2016.

\bibitem{Sager2009}
Sebastian Sager.
\newblock Reformulations and algorithms for the optimization of switching
  decisions in nonlinear optimal control.
\newblock {\em J Process Control}, 19(8):1238--1247, 2009.

\bibitem{SastryEtAl2013a}
Ramanarayan Vasudevan, Humberto Gonzalez, Ruzena Bajcsy, and S.~Shankar Sastry.
\newblock Consistent approximations for the optimal control of constrained
  switched systems---part 1: A conceptual algorithm.
\newblock {\em SIAM J Control Optim}, 51(6):4463--4483, 2013.

\bibitem{SastryEtAl2013b}
Ramanarayan Vasudevan, Humberto Gonzalez, Ruzena Bajcsy, and S.~Shankar Sastry.
\newblock Consistent approximations for the optimal control of constrained
  switched systems---part 2: An implementable algorithm.
\newblock {\em SIAM J Control Optim}, 51(6):4484--4503, 2013.

\bibitem{Williams1989}
A.~C. Williams.
\newblock Marginal values in mixed integer linear programming.
\newblock {\em Math Program}, 44(1, (Ser. A)):67--75, 1989.

\bibitem{Antsaklis2014}
Feng Zhu and Panos~J. Antsaklis.
\newblock Optimal control of hybrid switched systems: A brief survey.
\newblock {\em Discrete Event Dyn S}, 25(3):345--364, 2015.

\bibitem{Zuazua2011}
Enrique Zuazua.
\newblock Switching control.
\newblock {\em J Eur Math Soc}, 13:85--117, 2011.

\end{thebibliography}
\end{document}